\documentclass[12pt,twoside]{article}
\usepackage{amsmath,amsthm,amssymb,amscd,ascmac, amsfonts}
\usepackage{mathrsfs}
\usepackage{stmaryrd}
\usepackage{braket}
\usepackage{accents}
\usepackage{url}

\numberwithin{equation}{section}
\usepackage[dvips]{graphicx,color,psfrag}

\makeatletter

\renewcommand{\Re}{\mathrm{Re}\,}
\renewcommand{\Im}{\mathrm{Im}\,}

\newcommand{\sn}{\mathrm{sn}\,}
\newcommand{\cn}{\mathrm{cn}\,}
\newcommand{\dn}{\mathrm{dn}\,}
\newcommand{\cs}{\mathrm{cs}\,}
\newcommand{\ds}{\mathrm{ds}\,}
\newcommand{\ns}{\mathrm{ns}\,}

\renewcommand{\mod}{\,\mathrm{mod}\,}

\renewcommand{\gcd}[1]{\mathrm{gcd}\,(#1)}

\newcommand{\SL}{\mathrm{SL}}

\newtheorem*{multitheorem}{\variable@name}

\theoremstyle{definition}
\newcommand{\variable@name}{Theorem}
\newtheorem*{multiproclaim}{\variable@name}

\theoremstyle{plain}
\newtheorem{thm}{Theorem}
\newtheorem{prop}[thm]{Proposition}
\newtheorem{lem}[thm]{Lemma}
\newtheorem{cor}[thm]{Corollary}

\newtheorem{conj}[thm]{Conjecture}

\theoremstyle{definition}

\newtheorem{rmk}[thm]{Remark}

\begin{document}
\title{Some arithmetic properties of an elliptic Dedekind sum}
\author{Genki Shibukawa}
\date{
\small MSC classes\,:\,11F11, 11F20, 11M36, 33E05}
\pagestyle{plain}

\maketitle

\begin{abstract}
We give an explicit expression of the elliptic classical Dedekind sum which is a special case of multiple elliptic Dedekind sums introduced by Egami.  
We also determine the denominator of the rational part and zeros of the elliptic classical Dedekind sum. 
\end{abstract}

\section{Introduction}
Let $a$ and $b$ be relatively prime positive integers, then we set
\begin{align}
s(a;b)&:=
       \frac{1}{4b}\sum_{\nu =1}^{b-1}
       \cot{\left(\frac{\pi a\nu }{b}\right)}
       \cot{\left(\frac{\pi \nu }{b}\right)}, \\
s(a;1)&:=
       0, 
\end{align}
which is well known the classical Dedekind sum. 
Although the classical Dedekind sum has no closed form with some exceptions like 
$$
s(1;a)=\frac{(a-1)(a-2)}{12a}, \quad s(2;a)=\frac{(a-1)(a-5)}{24a}, 
$$
the following results hold. \\

\noindent
{(1)} {\bf{parity}}
\begin{equation}
\label{eq:parity}
s(-a;b)=-s(a;b),
\end{equation}
{(2)} {\bf{reduction}}
\begin{equation}
\label{eq:reduction}
s(a+b;b)=s(a;b),
\end{equation}
{(3)} {\bf{reciprocity}}
\begin{align}
\label{eq:reciprocity}
s(a;b)+s(b;a)=-\frac{1}{4}+\frac{a^{2}+b^{2}+1}{12ab}. 
\end{align}
More precisely, the classical Dedekind sum $s(a;b)$ has some arithmetic properties \cite{RG} : \\

\noindent
{(4)} {\bf{Determine denominator}}
\begin{equation}
\label{eq:denom}
(2b\cdot \gcd{3,b})\cdot s(a;b) \in \mathbb{Z},
\end{equation}

\noindent
{(5)} {\bf{Determine zeros}}
\begin{equation}
\label{eq:zero}
a^{2}+1\equiv 0 \quad (\,\mathrm{mod}\,{b}\,) \quad \Leftrightarrow \quad s(a;b)=0.
\end{equation}

For the classical Dedekind sum and its reciprocity law, there are various generalizations.  
In particular, Egami \cite{E} introduced an elliptic analogue of multiple Dedekind sums and gave its reciprocity law which is different from Sczech's elliptic Dedekind sum \cite{Sc}. 
After his job, Bayad, Fukuhara-Yui, Asano, Machide and et. al. have studied more generalization of the multiple elliptic Dedekind sums and their reciprocity laws. 
However, it seems that investigations of specializations of their results have not been yet even in an elliptic analogue of the classical Dedekind sum
\begin{align}
s_{\tau}(a;b)
  &:=
   \frac{1}{4b}\sum_{\substack{0\leq \mu ,\nu \leq b-1 \\ (\mu ,\nu ) \not=(0,0)}}
   (-1)^{\mu }
   \cs\left(2Ka\frac{\mu \tau +\nu }{b},k \right)
   \cs\left(2K\frac{\mu \tau +\nu }{b},k \right), \nonumber \\
s_{\tau}(a;1)
  &:=0, \nonumber
\end{align}
which we call {\it{elliptic classical Dedekind sum}}. 
This elliptic sum is desired to obtain more precise results like the classical case (\ref{eq:denom}), (\ref{eq:zero}). 
In this article, we give an explicit expression of the elliptic classical Dedekind sum and derive some arithmetic properties Theorem\,\ref{thm:key lemma} and Theorem\,\ref{thm:zero determine}.

The content of this paper is as follows. 
In Section\,2, we introduce the elliptic function $\cs(z,k)$ and list its fundamental properties according to the wolfram functions site \cite{Wo}. 
In Section\,3, we mention an elliptic analogue of the classical Dedekind sum and its known results. 
Section\,4 is the key part of this article. 
By using fundamental properties the elliptic function $\cs(z,k)$ and the elliptic classical Dedekind sum, we give an explicit expression of the elliptic classical Dedekind sum. 
In Section\,5 and 6, we derive fundamental properties of the rational part for the elliptic classical Dedekind sum. 
In particular, we determine the denominator of rational part and zeros. 
Finally, in Section\,7, we mention two problems related to our research.

\section{The elliptic function $\cs(z,k)$}
Throughout the paper, we denote the ring of rational integers by $\mathbb{Z}$, 
the field of real numbers by $\mathbb{R}$, the field of complex numbers by $\mathbb{C}$ and the upper half plane $\mathfrak{H}:=\{z \in \mathbb{C}\mid \Im{z}>0\}$. 
For $\tau \in \mathfrak{H}$, we put   
$$
e(x):=e^{2\pi \sqrt{-1}x}, \quad q:=e(\tau).
$$
First, we recall the Jacobi theta functions 
\begin{align}
\theta_{1}(z,\tau)
   :=&
   2\sum_{n=0}^{\infty}(-1)^{n}q^{\frac{1}{2}(n+\frac{1}{2})^{2}}\sin{((2n+1)\pi z)} \nonumber \\
   =&
   2q^{\frac{1}{8}}\sin{\pi z}\prod_{n=1}^{\infty}(1-q^{n})(1-q^{n}e(z))(1-q^{n}e(-z)), \nonumber \\
\theta_{2}(z,\tau)
   :=&
   2\sum_{n=0}^{\infty}q^{\frac{1}{2}(n+\frac{1}{2})^{2}}\cos{((2n+1)\pi z)} \nonumber \\
   =&
   2q^{\frac{1}{8}}\cos{\pi z}\prod_{n=1}^{\infty}(1-q^{n})(1+q^{n}e(z))(1+q^{n}e(-z)), \nonumber \\
\theta_{3}(z,\tau)
   :=&
   1+2\sum_{n=1}^{\infty}q^{\frac{1}{2}n^{2}}\cos{(2n\pi z)} \nonumber \\
   =&
   \prod_{n=1}^{\infty}(1-q^{n})(1+q^{n-\frac{1}{2}}e(z))(1+q^{n-\frac{1}{2}}e(-z)), \nonumber \\
\theta_{4}(z,\tau)
   :=&
   1+2\sum_{n=1}^{\infty}(-1)^{n}q^{\frac{1}{2}n^{2}}\cos{(2n\pi z)} \nonumber \\
   =&
   \prod_{n=1}^{\infty}(1-q^{n})(1-q^{n-\frac{1}{2}}e(z))(1-q^{n-\frac{1}{2}}e(-z)). \nonumber
\end{align}
Further we put 
$$
k=k(\tau):=\frac{\theta_{2}(0,\tau)^{2}}{\theta_{3}(0,\tau)^{2}},\quad 
\lambda =\lambda (\tau):=k(\tau)^{2},\quad 
K=K(\tau):=\frac{\pi}{2}\theta_{3}(0,\tau)^{2}
$$
and introduce the Jacobi elliptic functions
\begin{align}
\sn(2Kz,k)
   :=&
   \frac{\theta_{3}(0,\tau)}{\theta_{2}(0,\tau)}\frac{\theta_{1}(z,\tau)}{\theta_{4}(z,\tau)}, \nonumber \\
\cn(2Kz,k)
   :=&
   \frac{\theta_{4}(0,\tau)}{\theta_{2}(0,\tau)}\frac{\theta_{2}(z,\tau)}{\theta_{4}(z,\tau)}, \nonumber \\
\dn(2Kz,k)
   :=&
   \frac{\theta_{4}(0,\tau)}{\theta_{3}(0,\tau)}\frac{\theta_{3}(z,\tau)}{\theta_{4}(z,\tau)}. \nonumber 
\end{align}
As is well known, the Jacobi elliptic functions $\sn(2Kz,k)$, $\cn(2Kz,k)$ and $ \dn(2Kz,k)$ only depend on $\lambda (\tau)=k(\tau)^{2}$ (elliptic lambda function) that is a modular function of the modular subgroup
$$
\Gamma (2)
   :=
   \left\{
   \begin{pmatrix}
   a & b \\
   c & d
   \end{pmatrix} \in \SL_{2}(\mathbb{Z}) \,\Bigg\vert \,a\equiv d\equiv 1 \,(\mod 2),\,\,b\equiv c\equiv 0 \,(\mod 2)\right\}.
$$
Therefore under the following we restrict $\tau$ to the fundamental domain of $\Gamma (2)$ 
$$
\Gamma (2)\setminus \mathfrak{H}
   \simeq \left\{\tau \in \mathfrak{H} \,\Bigg\vert \, |\Re{\tau}|\leq 1, \, \left|\tau\pm \frac{1}{2}\right|\geq \frac{1}{2}\right\}.
$$

It is well known that 
\begin{align}
\label{eq:modular transform}
\lambda \left(-\frac{1}{\tau}\right)=1-\lambda (\tau), \quad \lambda \left(\tau +1\right)=\frac{\lambda (\tau )}{\lambda (\tau )-1}. 
\end{align}
In particular, since
$$
\lambda (\sqrt{-1})=\lambda \left(\frac{-1}{\sqrt{-1}}\right)=1-\lambda (\sqrt{-1}), 
$$
we have
\begin{equation}
\label{eq:special values of k}
\lambda (\sqrt{-1})=\frac{1}{2}.
\end{equation}

The elliptic function $\cs(2Kz,k)$ is defined by
$$
\cs(2Kz,k):=\frac{\cn(2Kz,k)}{\sn(2Kz,k)},
$$
which is regarded as an elliptic analogue of $\cot{(\pi z)}=\frac{\cos{(\pi z)}}{\sin{(\pi z)}}$. 
According to the wolfram functions site \cite{Wo}, we list fundamental properties of $\cs$. 
\begin{lem}
{\rm{(1)}}\;(parity) 
\begin{equation}
\label{eq:parity of cs}
\cs(-2Kz,k)=-\cs(2Kz,k).
\end{equation}
\url{http://functions.wolfram.com/EllipticFunctions/JacobiCS/04/02/01/}\\
\noindent
{\rm{(2)}}\;(periodicity) For any $\mu,\nu \in \mathbb{Z}$, 
\begin{align}
\label{eq:periodicity of cs}
\cs(2K(z+\mu \tau+\nu ),k)&=(-1)^{\mu }\cs(2Kz,k), 
\end{align}
\url{http://functions.wolfram.com/EllipticFunctions/JacobiCS/04/02/03/}\\
\noindent
{\rm{(3)}}\;(Laurent expansion at $z=0$)  
\begin{align}
\label{eq:Laurent expansion of cs}
2K\cs(2Kz,k)&=\frac{1}{z}-\left(\frac{1}{3}-\frac{1}{6}\lambda \right)(2K)^{2}z-\left(\frac{1}{45}-\frac{1}{45}\lambda -\frac{7}{360}\lambda ^{2}\right)(2K)^{4}z^{3}-\cdots . 
\end{align}
\url{http://functions.wolfram.com/EllipticFunctions/JacobiCS/06/01/01/}\\
\noindent
{\rm{(4)}}\;(Derivation)
\begin{equation}
\frac{\partial }{\partial z}2K\cs(2Kz,k)
   =
   -2K\ds(2Kz,k)2K\ns(2Kz,k).
\end{equation}
\url{http://functions.wolfram.com/EllipticFunctions/JacobiCS/20/01/01/}\\
\noindent
{\rm{(5)}}\;(Relation between the Weierstrass $\wp$ function)
\begin{equation}
(2K\cs(2Kz,k))^{2}=\wp(z,\tau)-\wp\left(\frac{1}{2},\tau \right).
\end{equation}
Here, $\wp(z,\tau)$ is the Weierstrass $\wp$ function defined by
$$
\wp(z,\tau):=\frac{1}{z^{2}}+\sum_{\substack{n,m \in \mathbb{Z} \\ (n,m) \not=(0,0)}}\left\{\frac{1}{(m+n\tau +z)^{2}}-\frac{1}{{(m+n\tau)}^{2}}\right\}.
$$
\url{http://functions.wolfram.com/EllipticFunctions/JacobiCS/27/02/07/}\\
\noindent
{\rm{(6)}}\;(trigonometric degeneration)
\begin{align}
\label{eq:degeneration of k}
\lim_{\tau \to \sqrt{-1}\infty}
   k(\tau)&=0, \\
\label{eq:degeneration of K}
\lim_{\tau \to \sqrt{-1}\infty}
   2K(\tau)&=\pi, \\   
\label{eq:degeneration of cs}
\lim_{\tau \to \sqrt{-1}\infty}
   \cs(2K(z+w\tau),k)
   &=\begin{cases}
   (-1)^{w} \cot(\pi z) & (w \in \mathbb{Z}) \\
   -(-1)^{\lfloor w\rfloor } i & (w \not\in \mathbb{Z})
   \end{cases}.
\end{align}
Here, $\lfloor w \rfloor$ denotes the greatest integer not exceeding $w$. 
\end{lem}

\begin{rmk}
{\rm{(1)}} Egami and others use 
$$
\varphi (\tau, z):=\sqrt{\wp(z,\tau)-\wp\left(\frac{1}{2},\tau \right)}=\frac{1}{z}+O(z) \quad (z \to 0)
$$
instead of $2K\cs(2Kz,k)$. 
However, Egami and others did not mention that $\varphi (\tau, z)$ is the Jacobi elliptic function $2K\cs(2Kz,k)$ exactly. \\
{\rm{(2)}} R.\,Sczech \cite{Sc} introduced another elliptic Dedekind sum. 
He considered a real analytic Eisenstein series
$$
G(s,z;\tau):=\sum_{m ,n \in \mathbb{Z}}
\frac{\overline{m\tau +n+z}}{|m\tau +n+z|^{2s}}
$$
for $\mathrm{Re}(s)>1$ and by analytic continuation for other values of the complex number $s$. 
In particular, $G(1,z;\tau)$ is real analytic and doubly periodic for $z$
$$
G(1,z+1;\tau)=G(1,z+\tau;\tau)=G(1,z;\tau), 
$$
which is regarded as another elliptic analogue of cotangent function. 
Actually, $G(1,z;\tau)$ has the following explicit expression 
$$
G(1,z;\tau)
   =
   \zeta(z,\tau)
   -zG _{2}(\tau)
   +\frac{2\pi i}{\tau -\overline{\tau}}(z-\overline{z}), 
$$
where $\zeta(z,\tau)$ is Weierstrass $\zeta$ function 
$$
\zeta(z,\tau)
   :=
   \frac{1}{z}+\sum_{\substack{n,m \in \mathbb{Z} \\ (n,m) \not=(0,0)}}\left\{\frac{1}{m+n\tau +z}-\frac{1}{m+n\tau}+\frac{z}{{(m+n\tau)}^{2}}\right\}
$$
and $G_{2}(\tau)$ is Eisenstein series of weight $2$
$$
G_{2}(\tau)
   :=
   \frac{\pi ^{2}}{3}
   -\sum_{n=1}^{\infty}
      \frac{8\pi ^{2}q^{n}}{(1-q^{n})^{2}}.
$$
\end{rmk}

\section{The elliptic classical Dedekind sum}
Let $a$ and $b$ are relatively prime positive integers. 
According to Egami \cite{E}, we introduce the elliptic classical Dedekind sum by 
\begin{align}
s_{\tau}(a;b)
  &:=
   \frac{1}{4b}\sum_{\substack{0\leq \mu ,\nu \leq b-1 \\ (\mu ,\nu ) \not=(0,0)}}
   (-1)^{\mu }
   \cs\left(2Ka\frac{\mu \tau +\nu }{b},k \right)
   \cs\left(2K\frac{\mu \tau +\nu }{b},k \right), \\
s_{\tau}(a;1)
  &:=0. 
\end{align}
It is regarded as an elliptic analogue of the classical Dedekind sum
\begin{align}
s(a;b)&:=
       \frac{1}{4b}\sum_{\nu =1}^{b-1}
       \cot{\left(\frac{\pi a\nu }{b}\right)}
       \cot{\left(\frac{\pi \nu }{b}\right)}, \nonumber \\
s(a;1)&:=
       0. \nonumber
\end{align}
For convenience, we introduce the following notations, 
\begin{align}
R(a;b)
   &:=\frac{a^{2}+b^{2}+1}{4ab}, \nonumber \\
U^{o}
   &:=\{(a;b) \in \mathbb{Z}\setminus \{0\}\times \mathbb{Z}_{>0}
      \mid \mathrm{gcd}(a,b)=1,\; \text{$a+b$ is odd}\}.  \nonumber
\end{align}
For the elliptic classical Dedekind sum, the following properties are known. 
\begin{thm}
\label{thm:thm1}
{\rm{(1)}} (parity)
\begin{equation}
\label{eq:parity of the elliptic Dedekind sum}
s_{\tau}(-a;b)=-s_{\tau}(a;b).
\end{equation}
{\rm{(2)}} (even reduction)
\begin{equation}
\label{eq:even reduction of the elliptic Dedekind sum}
s_{\tau}(a+2b;b)=s_{\tau}(a;b).
\end{equation}
{\rm{(3)}} (inversion formula) 
If $b$ is odd and $aa^{\prime}\equiv \pm 1\,(\mod{b}\,)$, or $b$ is even and $aa^{\prime}\equiv \pm 1\,(\mod{2b}\,)$, then
\begin{equation}
\label{eq:inversion formula of the elliptic Dedekind sum}
s_{\tau}(a;b)=\pm s_{\tau}(a^{\prime};b). 
\end{equation}
{\rm{(4)}} (reciprocity)
\begin{equation}
\label{eq:reciprocity of the elliptic Dedekind sum}
s_{\tau}(a;b)+s_{\tau}(b;a)=R(a;b)\left(\frac{1}{3}-\frac{1}{6}\lambda (\tau)\right) \quad ((a;b) \in U^{o}).
\end{equation}
{\rm{(5)}} (rationality) 
For any $(a;b) \in U^{o}$, there exists a unique rational number $Q(a;b)$ such that 
\begin{equation}
\label{eq:values of the elliptic Dedekind sum}
s_{\tau}(a;b) =Q(a;b)\left(\frac{1}{3}-\frac{1}{6}\lambda (\tau)\right).
\end{equation}
{\rm{(6)}} (degeneration)
\begin{equation}
\label{eq:degeneration of the elliptic Dedekind sum}
\lim_{\tau \to \sqrt{-1}\infty}s_{\tau}(a;b)=s(a;b)+\frac{1}{4}S(a;b).
\end{equation}
Here, $S(a;b)$ is the Hardy-Berndt sum defined by 
$$
S(a;b):=\sum_{\mu =1}^{b-1}(-1)^{\mu +1+\left\lfloor \frac{a\mu }{b}\right\rfloor}.
$$
\end{thm}
\begin{proof}
Actually, (\ref{eq:parity of the elliptic Dedekind sum}) and (\ref{eq:even reduction of the elliptic Dedekind sum}) follow from (\ref{eq:parity of cs}) and (\ref{eq:periodicity of cs}) respectively.

For (\ref{eq:inversion formula of the elliptic Dedekind sum}), in the case that $b$ is odd,  
\begin{align}
s_{\tau}(a;b)
   &=
   \frac{1}{4b}\sum_{\substack{0\leq \mu ,\nu \leq b-1 \\ (\mu ,\nu ) \not=(0,0)}}
   (-1)^{\mu }
   \cs\left(2Ka\frac{2a^{\prime}\mu \tau +2a^{\prime}\nu }{b},k \right)
   \cs\left(2K\frac{2a^{\prime}\mu \tau +2a^{\prime}\nu }{b},k \right) \nonumber \\
   &=
   \frac{1}{4b}\sum_{\substack{0\leq \mu ,\nu \leq b-1 \\ (\mu ,\nu ) \not=(0,0)}}
   (-1)^{\mu }
   \cs\left(2K2aa^{\prime}\frac{\mu \tau +\nu }{b},k \right)
   \cs\left(2Ka^{\prime}\frac{2\mu \tau +2\nu }{b},k \right) \nonumber \\
   &=
   \frac{1}{4b}\sum_{\substack{0\leq \mu ,\nu \leq b-1 \\ (\mu ,\nu ) \not=(0,0)}}
   (-1)^{\mu }
   \cs\left(2Ka^{\prime}\frac{2\mu \tau +2\nu }{b},k \right)
   \cs\left(\pm 2K\frac{2\mu \tau +2\nu }{b},k \right) \nonumber \\
   &=
   \pm s_{\tau}(a^{\prime};b). \nonumber
\end{align}
Similarly, in the case that $b$ is even, 
\begin{align}
s_{\tau}(a;b)
   &=
   \frac{1}{4b}\sum_{\substack{0\leq \mu ,\nu \leq b-1 \\ (\mu ,\nu ) \not=(0,0)}}
   (-1)^{\mu }
   \cs\left(2Ka\frac{a^{\prime}\mu \tau +a^{\prime}\nu }{b},k \right)
   \cs\left(2K\frac{a^{\prime}\mu \tau +a^{\prime}\nu }{b},k \right) \nonumber \\
   &=
   \frac{1}{4b}\sum_{\substack{0\leq \mu ,\nu \leq b-1 \\ (\mu ,\nu ) \not=(0,0)}}
   (-1)^{\mu }
   \cs\left(2Kaa^{\prime}\frac{\mu \tau +\nu }{b},k \right)
   \cs\left(2Ka^{\prime}\frac{\mu \tau +\nu }{b},k \right) \nonumber \\
   &=
   \frac{1}{4b}\sum_{\substack{0\leq \mu ,\nu \leq b-1 \\ (\mu ,\nu ) \not=(0,0)}}
   (-1)^{\mu }
   \cs\left(2Ka^{\prime}\frac{\mu \tau +\nu }{b},k \right)
   \cs\left(\pm 2K\frac{\mu \tau +\nu }{b},k \right) \nonumber \\
   &=
   \pm s_{\tau}(a^{\prime};b). \nonumber
\end{align}

The reciprocity (\ref{eq:reciprocity of the elliptic Dedekind sum}) is a specialization of the Egami's reciprocity \cite{E} and Lemma\,3.1 in \cite{FY}. 
Unfortunately, Egami's original statement (Theorem\,1 in \cite{E}) is incorrect, which is pointed out by Fukuhara-Yui \cite{FY}. 
Hence, we refer the correct result from Lemma\,3.1 in \cite{FY}.\\
Rationality of $s_{\tau}(a;b)$ (\ref{eq:values of the elliptic Dedekind sum}) follows from (\ref{eq:parity of the elliptic Dedekind sum}), (\ref{eq:even reduction of the elliptic Dedekind sum}) and (\ref{eq:reciprocity of the elliptic Dedekind sum}) immediately.

The degenerate limit (\ref{eq:degeneration of the elliptic Dedekind sum}) corresponds to trigonometric degeneration (\ref{eq:degeneration of cs}). 
Actually, 
\begin{align}
\lim_{\tau \to \sqrt{-1}\infty}s_{\tau}(a;b)
   &=
   \lim_{\tau \to \sqrt{-1}\infty}
   \frac{1}{4b}\sum_{\nu =1}^{b-1}
   (-1)^{\mu }
   \cs\left(2Ka\frac{\nu }{b},k \right)
   \cs\left(2K\frac{\nu }{b},k \right) \nonumber \\
   & \quad +
   \lim_{\tau \to \sqrt{-1}\infty}
   \frac{1}{4b}\sum_{\substack{0\leq \mu ,\nu \leq b-1 \\ \mu \not=0}}
   (-1)^{\mu }
   \cs\left(2Ka\frac{\mu \tau +\nu }{b},k \right)
   \cs\left(2K\frac{\mu \tau +\nu }{b},k \right) \nonumber \\
   &=
   \frac{1}{4b}\sum_{\nu =1}^{b-1}
   (-1)^{\mu }
   \cot\left(\frac{\pi a\nu }{b},k \right)
   \cot\left(\frac{\pi \nu }{b},k \right) \nonumber \\
   & \quad +
   \frac{1}{4b}\sum_{\substack{0\leq \mu ,\nu \leq b-1 \\ \mu \not=0}}
   (-1)^{\mu }(-i)^{2}
   (-1)^{\left\lfloor \frac{a \mu }{b} \right \rfloor} \nonumber \\
   &=
   s(a;b)+\frac{1}{4b}bS(a;b). \nonumber 
\end{align}

We point out the values of $s_{\tau}(a;b)$ on $U^{o}$ are determined by (\ref{eq:parity of the elliptic Dedekind sum}), (\ref{eq:even reduction of the elliptic Dedekind sum}), (\ref{eq:reciprocity of the elliptic Dedekind sum}) and the Euclidean algorithm exactly. 
Hence we need not recall the original definition of $\cs(z,k)$ to evaluate $s_{\tau}(a;b)$ on $U^{o}$. 
As a corollary of (\ref{eq:values of the elliptic Dedekind sum}), we have the following result. 
\end{proof}

\begin{cor}
\begin{align}
\label{eq:modular trans 2}
s_{-\frac{1}{\tau }}(a;b)
   &=Q(a;b)\left(\frac{1}{6}+\frac{1}{6}\lambda (\tau)\right), \\
\label{eq:modular trans 3}
s_{\tau +1}(a;b)
   &=Q(a;b)\frac{\lambda (\tau)-2}{6(\lambda (\tau)-1)}, \\
\label{eq:modular trans 4}
s_{-\frac{1}{\tau +1}}(a;b)
   &=Q(a;b)\frac{1-2\lambda (\tau)}{6(\lambda (\tau)-1)}, \\
\label{eq:modular trans 5}
s_{\frac{\tau -1}{\tau}}(a;b)
   &=Q(a;b)\frac{\lambda (\tau)+1}{6\lambda (\tau)}, \\
\label{eq:modular trans 6}
s_{\frac{\tau}{\tau +1}}(a;b)
   &=Q(a;b)\frac{2\lambda (\tau)-1}{6\lambda (\tau)}.
\end{align}
In particular, 
\begin{equation}
s_{\frac{-1+\sqrt{-1}}{2}}(a;b)=s_{\frac{1+\sqrt{-1}}{2}}(a;b)=0.
\end{equation}
\end{cor}
\begin{proof}
From modular transform (\ref{eq:modular transform}), 
\begin{equation}
\lambda \left(-\frac{1}{\tau +1}\right)
   =\frac{1}{1-\lambda (\tau)}, \quad 
\lambda \left(\frac{\tau -1}{\tau }\right)
   =\frac{\lambda (\tau)-1}{\lambda (\tau)}, \quad 
\lambda \left(\frac{\tau }{\tau +1}\right)
   =\frac{1}{\lambda (\tau)}. \nonumber
\end{equation}
Hence we have (\ref{eq:modular trans 2}) - (\ref{eq:modular trans 6}). 
Recalling a special values of $\lambda $ (\ref{eq:special values of k}), 
\begin{align}
s_{\frac{-1+\sqrt{-1}}{2}}(a;b)
   &=
   s_{-\frac{1}{\sqrt{-1}+1}}(a;b)
   =
   Q(a;b)\left(\frac{1-2\lambda (\sqrt{-1})}{6(\lambda (\sqrt{-1})-1)}\right)
   =0, \nonumber \\
s_{\frac{1+\sqrt{-1}}{2}}(a;b)
   &=
   s_{\frac{\sqrt{-1}}{\sqrt{-1}+1}}(a;b)
   =
   Q(a;b)\left(\frac{2\lambda (\sqrt{-1})-1}{6\lambda (\sqrt{-1})}\right)
   =0. \nonumber
\end{align}
\end{proof}
Under the following we assume $\tau \not=\frac{-1+\sqrt{-1}}{2}$ or $\frac{1+\sqrt{-1}}{2}$. 
In these cases, since $\frac{1}{3}-\frac{1}{6}\lambda (\tau)\not=0$, our elliptic classical Dedekind sum $s_{\tau}(a;b)$ on $U^{o}$ is equal to the rational part $Q(a;b)$ up to the constant factor $\frac{1}{3}-\frac{1}{6}\lambda (\tau )$. 
Under the following sections, we assume $(a;b) \in U^{o}$. 
\begin{thm}
\label{thm:rational part properties}
{\rm{(1)}} (parity)
\begin{equation}
\label{eq:parity of the Q}
Q(-a;b)=-Q(a;b).
\end{equation}
{\rm{(2)}} (even reduction)
\begin{equation}
\label{eq:even reduction of the Q}
Q(a+2b;b)=Q(a;b).
\end{equation}
{\rm{(3)}} (inversion formula) 
If $b$ is odd and $a$, $a^{\prime}$ are even such that $aa^{\prime}\equiv \pm 1\,(\mod{b}\,)$, or $b$ is even and $a$, $a^{\prime}$ are even such that $aa^{\prime}\equiv \pm 1\,(\mod{2b}\,)$, then
\begin{equation}
\label{eq:inversion formula of the Q}
Q(a;b)=\pm Q(a^{\prime};b).
\end{equation}
{\rm{(4)}} (reciprocity)
\begin{equation}
\label{eq:reciprocity of the Q}
Q(a;b)+Q(b;a)=R(a;b) \quad ((a;b) \in U^{o}).
\end{equation}
\end{thm}
Actually, (\ref{eq:parity of the Q}), (\ref{eq:even reduction of the Q}), (\ref{eq:inversion formula of the Q}) and (\ref{eq:reciprocity of the Q}) correspond to (\ref{eq:parity of the elliptic Dedekind sum}), (\ref{eq:even reduction of the elliptic Dedekind sum}), (\ref{eq:inversion formula of the elliptic Dedekind sum}) and (\ref{eq:reciprocity of the elliptic Dedekind sum}) respectively. 
We remark that $Q(a;b)$ on $U^{o}$ are determined by (\ref{eq:parity of the Q}), (\ref{eq:even reduction of the Q}) and (\ref{eq:reciprocity of the Q}) exactly. Hence, using (\ref{eq:parity of the Q}), (\ref{eq:even reduction of the Q}) and (\ref{eq:reciprocity of the Q}), we give tables of $4bs_{\sqrt{-1}}(a;b)=bQ(a;b)$. 

\begin{table}[h]
\begin{center}
  \begin{tabular}{|c||c|c|c|c|c|c|c|c|c|c|c|} \hline
   $b\setminus a$ & 2 & 4 & 6 & 8 & 10 & 12 & 14 & 16 & 18 & 20 & 22\\ \hline \hline
   1 & 0 & 0 & 0 & 0 & 0 & 0 & 0 & 0 & 0 & 0 & 0\\ \hline
   3 & 4 & -4 & $\ast$ & 4 & -4 & $\ast$ & 4 & -4 & $\ast$ & 4 & -4\\ \hline
   5 & 0 & 12 & -12 & 0 & $\ast$ & 0 & 12 & -12 & 0 & $\ast$ & 0 \\ \hline
   7 & 12 & 12 & 24 & -24 & -12 & -12 & $\ast$ & 12 & 12 & 24 & -24\\ \hline
   9 & 4 & -4 & $\ast$ & 40 & -40 & $\ast$ & 4 & -4 & $\ast$ & 40 & -40 \\ \hline
   11 & 24 & -12 & 24 & 12 & 60 & -60 & -12 & -24 & 12 & -24 & $\ast$ \\ \hline
   13 & 12 & 36 & -12 & 0 & 36 & 84 & -84 & -36 & 0 & 12 & -36 \\ \hline
   15 & 40 & 32 & $\ast$ & 40 & $\ast$ & $\ast$ & 112 & -112 & $\ast$ & $\ast$ & -40 \\ \hline
   17 & 24 & 0 & -36 & -24 & 48 & 48 & 36 & 144 & -144 & -36 & -48 \\ \hline
   19 & 60 & -12 & 72 & 24 & 60 & 24 & 12 & 72 & 180 & -180 & -72 \\ \hline
   21 & 40 & 68 & $\ast$ & 4 & -40 & $\ast$ & $\ast$ & 68 & $\ast$ & 220 & -220 \\ \hline
  \end{tabular}
\end{center}
\end{table}

\begin{table}[h]
\begin{center}
  \begin{tabular}{|c||c|c|c|c|c|c|c|c|c|c|c|} \hline
   $b\setminus a$ & 1 & 3 & 5 & 7 & 9 & 11 & 13 & 15 & 17 & 19 & 21 \\ \hline \hline
   2 & 1.5 & -1.5 & 1.5 & -1.5 & 1.5 & -1.5 & 1.5 & -1.5 & 1.5 & -1.5 & 1.5\\ \hline
   4 & 4.5 & 7.5 & -7.5 & -4.5 & 4.5 & 7.5 & -7.5 & -4.5 & 4.5 & 7.5 & -7.5 \\ \hline
   6 & 9.5 & $\ast $ & 17.5 & -17.5 & $\ast $ & -9.5 & 9.5 & $\ast $ & 17.5 & -17.5 & $\ast $ \\ \hline
   8 & 16.5 & -4.5 & 4.5 & 31.5 & -31.5 & -4.5 & 4.5 & -16.5 & 16.5 & -4.5 & 4.5 \\ \hline
   10 & 25.5 & 22.5 & $\ast $ & 22.5 & 49.5 & -49.5 & -22.5 & $\ast $ & -22.5 & -25.5 & 25.5 \\ \hline
   12 & 36.5 & $\ast $ & 8.5 & 27.5 & $\ast $ & 71.5 & -71.5 & $\ast $ & -27.5 & -8.5 & $\ast $ \\ \hline
   14 & 49.5 & -1.5 & -22.5 & $\ast $ & 1.5 & 22.5 & 97.5 & -97.5 & -22.5 & -1.5 & $\ast $ \\ \hline
   16 & 64.5 & 43.5 & 52.5 & -16.5 & 16.5 & 43.5 & 52.5 & 127.5 & -127.5 & -52.5 & -43.5 \\ \hline
   18 & 81.5 & $\ast $ & 17.5 & -17.5 & $\ast $ & -9.5 & 9.5 & $\ast $ & 161.5 & -161.5 & $\ast $ \\ \hline
   20 & 100.5 & 7.5 & $\ast $ & -52.5 & 4.5 & 55.5 & -7.5 & $\ast $ & 52.5 & 199.5 & -199.5 \\ \hline
   22 & 121.5 & 70.5 & 25.5 & 94.5 & 25.5 & $\ast $ & 73.5 & 70.5 & 73.5 & 94.5 &  241.5 \\ \hline
\end{tabular}
\caption{$bQ(a;b)$}
\end{center}
\end{table}

\newpage

\section{An explicit formula of the rational part $Q(a;b)$}
\begin{thm}
\label{thm:main result}
\begin{align}
\label{eq:main result}
Q(a;b)
   =3\left(s(a;b)+\frac{1}{4}S(a;b)\right).
\end{align}
\end{thm}
\begin{proof}
From (\ref{eq:values of the elliptic Dedekind sum}), 
$$
s_{\tau}(a;b)
   =Q(a;b)\left(\frac{1}{3}-\frac{1}{6}\lambda (\tau)\right).
$$
Since $Q(a;b)$ does not depend on $\tau$, we have
\begin{align}
Q(a;b)
   =
   \lim_{\tau \to \sqrt{-1}\infty}
   \frac{s_{\tau}(a;b)}{\frac{1}{3}-\frac{1}{6}\lambda (\tau)}. \nonumber
\end{align}
Recalling trigonometric degenerations (\ref{eq:degeneration of k}) and (\ref{eq:degeneration of the elliptic Dedekind sum}), we obtain the conclusion. 
\end{proof}
The Hardy-Berndt sum is written by the classical Dedekind sum. 
Actually, Sitaramachandra\,Rao prove the following formula. 
\begin{lem}[Sitaramachandra\,Rao \cite{Si}]
If $(a;b) \in U^{o}$, then 
\begin{equation}
\label{eq:Key lemma}
S(a;b)=8s(a;2b)+8s(2a;b)-20s(a;b).
\end{equation}
\end{lem}
Using this formula (\ref{eq:Key lemma}), we have 
\begin{equation}
\label{eq:another expression}
Q(a;b)
   =6\left(s(a;2b)+s(2a;b)-2s(a;b)\right).
\end{equation}

\section{Denominator}
We determine the denominator of $Q(a;b)$. 
\begin{thm}
\label{thm:key lemma}
For any $(a;b) \in U^{o}$, there exists an integer $M(a;b) \in \mathbb{Z}$ such that
\begin{equation}
bQ(a;b)
   =
   \frac{a(1-3b)}{2}+M(a;b).
\end{equation}
In particular, 
\begin{equation}
bQ(a;b)\in 
   \begin{cases}
   \frac{1}{2}+\mathbb{Z} & \text{(if $a$ is odd and $b$ is even)} \\
   \mathbb{Z} & \text{(if $a$ is even and $b$ is odd)}
   \end{cases}.
\end{equation}
\end{thm}
\begin{proof}
It is well known that the classical Dedekind sum has the following expression \cite{RG}. 
\begin{align}
\label{eq:explicit expression of Dedekind sum}
s(a;b)
   &=
   \sum_{\nu =1}^{b-1}
      \left(\frac{a\nu }{b}-\left\lfloor\frac{a\nu }{b}\right\rfloor-\frac{1}{2}\right)
      \left(\frac{\nu }{b}-\frac{1}{2}\right). 
\end{align}
From (\ref{eq:another expression}) and (\ref{eq:explicit expression of Dedekind sum}), we have 
\begin{equation}
Q(a;b)
   =ab+\frac{a(1-3b)}{2b}
         +\frac{3}{b}
         \sum_{\nu =1}^{2b-1}
         \left\{\left\lfloor\frac{a\nu }{2b}\right\rfloor
         (b-\nu )\right\}
         +\frac{3}{b}\sum_{\nu =1}^{b-1}
         \left\{
         \left(\left\lfloor\frac{2a\nu }{b}\right\rfloor  
         -2\left\lfloor\frac{a\nu }{b}\right\rfloor \right) (b-2\nu)
         \right\}. \nonumber
\end{equation}
Hence if we put
$$
M(a;b)
   :=
   ab^{2}
   +3\sum_{\nu =1}^{2b-1}
   \left\{\left\lfloor\frac{a\nu }{2b}\right\rfloor
   (b-\nu )\right\}
   +3\sum_{\nu =1}^{b-1}
   \left\{
   \left(\left\lfloor\frac{2a\nu }{b}\right\rfloor  
   -2\left\lfloor\frac{a\nu }{b}\right\rfloor \right) (b-2\nu)
   \right\}
$$
then
$$
bQ(a;b)
   =
   \frac{a(1-3b)}{2}+M(a;b).
$$
\end{proof}

As a corollary of Theorem\,\ref{thm:main result} and Lemma\,\ref{thm:key lemma}, we obtain some properties for the integral part $M(a;b)$. 
\begin{cor}
{\rm{(1)}} (parity)
\begin{equation}
\label{eq:parity of the M}
M(-a;b)=-M(a;b).
\end{equation}
{\rm{(2)}} (even reduction)
\begin{equation}
\label{eq:even reduction of M}
M(a+2b;b)=M(a;b)+b(1-3b).
\end{equation}
{\rm{(3)}} (reciprocity)
\begin{equation}
\label{eq:reciprocity of M}
aM(a;b)+bM(b;a)=\frac{1-a^{2}-b^{2}+6ab(a+b)}{4}.
\end{equation}
\end{cor}

\section{Zeros}
We determine the denominator of zeros for $Q(a;b)$. 
\begin{prop}
\label{thm:prop zero}
If $a$ is even and $b$ is odd such that 
$$
a^{2}\equiv -1 \quad (\mod{b}\,), 
$$
then we have 
\begin{equation}
Q(a;b)=0. 
\end{equation}
\end{prop}
\begin{proof}
From assumptions of Proposition\,\ref{thm:prop zero}, if $a^{\prime}=a$ then 
$$
aa^{\prime}\equiv -1 \quad (\mod{b}\,).
$$
We recall {\rm{(3)}} of Theorem\,\ref{thm:rational part properties} and have 
$$
Q(a;b)=-Q(a^{\prime};b)=-Q(a;b).
$$
\end{proof}
\begin{thm}
\label{thm:zero determine}
Let $a$ be even and $b$ be odd. 
$$
Q(a;b) \in \mathbb{Z} \quad \Rightarrow \quad a^{2}\equiv -1 \quad (\mod{b}\,). 
$$
In particular, 
$$
Q(a;b)=0 \quad \Leftrightarrow \quad a^{2}\equiv -1 \quad (\mod{b}\,). 
$$
\end{thm}
\begin{proof}
Since $a$ is even and $b$ is odd, 
$2aQ(b;a)=b(1-3a)+2M(b;a) \in \mathbb{Z}$ and 
$2abQ(b;a) \in b\mathbb{Z}$. 
If $Q(a;b) \in \mathbb{Z}$, then $4abQ(a;b) \in b\mathbb{Z}$. 
The reciprocity (\ref{eq:reciprocity of the Q}) multiplied by $4ab$ gives 
\begin{equation}
\label{eq:cal1}
4abQ(a;b)+4abQ(b;a)=a^{2}+b^{2}+1.
\end{equation}
Hence we have
$$
a^{2}+1\equiv 0 \quad (\mathrm{mod}\,b).
$$
From Proposition\,\ref{thm:prop zero}, $Q(a;b)=0$.  
\end{proof}
Next, we construct zeros pair $(a;b)$ of $Q(a;b)$ explicitly. 
Let $M$ be a non negative integer and $N$ be a positive integer. 
We consider the following positive integer sequence defined by 
$$
P_{M+2}^{(N)}=NP_{M+1}^{(N)}+P_{M}^{(N)}, \quad P_{0}^{(N)}=0, \quad P_{1}^{(N)}=1, 
$$
which is a generalization Fibonacci sequence 
$$
F_{M+2}=F_{M+1}+F_{M}, \quad F_{0}=0, \quad F_{1}=1
$$
or Pell sequence 
$$
P_{M+2}=2P_{M+1}+P_{M}, \quad P_{0}=0, \quad P_{1}=1.
$$
We remark that the Cassini type formula 
$$
P_{M+1}^{(N)}P_{M-1}^{(N)}-{P_{M}^{(N)}}^{2}=(-1)^{M}
$$
holds for this sequence $\{P_{M}^{(N)}\}_{M=0,1,2,\ldots}$. 
Therefore by applying Theorem\,\ref{thm:rational part properties} {\rm{(3)}} to $\{P_{M}^{(N)}\}_{M=0,1,2,\ldots}$, we construct some zeros pairs explicitly. 
\begin{thm}
{\rm{(1)}}We have 
$$
P_{2m+1}^{(2n)} \in 1+2\mathbb{Z}, \quad P_{2m}^{(2n)} \in 2\mathbb{Z}
$$
and 
\begin{align}
Q(P_{2m}^{(2n)};P_{2m\pm 1}^{(2n)})&=0, \\
Q(P_{2m\pm 1}^{(2n)};P_{2m}^{(2n)})&=Q(P_{2m\mp 1}^{(2n)};P_{2m}^{(2n)}). 
\end{align}
{\rm{(2)}} We have 
$$
P_{3m\pm 1}^{(2n+1)} \in 1+2\mathbb{Z}, \quad P_{3m}^{(2n+1)} \in 2\mathbb{Z}
$$
and 
\begin{align}
Q(P_{3m}^{(2n+1)};P_{3m\pm 1}^{(2n+1)})&=(-1)^{m}Q(P_{3m}^{(2n+1)};P_{3m\pm 1}^{(2n+1)}), \\
Q(P_{3m\pm 1}^{(2n+1)};P_{3m}^{(2n+1)})&=(-1)^{m}Q(P_{3m\mp 1}^{(2n)};P_{3m}^{(2n)}). 
\end{align}
In particular, we obtain
\begin{equation}
Q(P_{6m+3}^{(2n+1)};P_{6m+3\pm 1}^{(2n+1)})=0.
\end{equation}
\end{thm}


\section{Concluding remarks}
We raise two problems related to our investigation. 
First, we desire to give more precisely result than Lemma\,\ref{thm:key lemma}. 
Actually, from the Table 2 of $4bs_{\sqrt{-1}}(a;b)=bQ(a;b)$, we consider the following conjecture. 
\begin{conj}
If $a$ is even and $b$ is odd then 
$$
4bs_{\sqrt{-1}}(a;b)=bQ(a;b) \in 4\mathbb{Z}.
$$ 
More precisely, 
$$
(6m\pm 1)Q(a;6m\pm 1) \in 12\mathbb{Z}.
$$
\end{conj}

\newpage

\begin{center}
\begin{table}[htb]
  \begin{tabular}{|c||c|c|c|c|c|c|c|c|c|c|c|c|c|c|c|c|} \hline
   $a\setminus b$ & 2 & 4 & 6 & 8 & 10 & 12 & 14 & 16 & 18 & 20 & 22 & 24 & 26 & 28 & 30 & 32 \\ \hline \hline
   3 & 1 &   &   &   &   &   &   &   &   &   &   &   &   &   &   &      \\ \hline
   5 & 0 & 3 &   &   &   &   &   &   &   &   &   &   &   &   &   &      \\ \hline
   7 & 3 & 3 & 6 &   &   &   &   &   &   &   &   &   &   &   &   &      \\ \hline
   9 & 1 & -1 & $\ast$ & 10 &   &   &   &   &   &   &   &   &   &   &   &      \\ \hline
   11 & 6 & -3 & 6 & 3 & 15 &   &   &   &   &   &   &   &   &   &   &   \\ \hline
   13 & 3 & 9 & -3 & 0 & 9 & 21 &   &   &   &   &   &   &   &   &   &  \\ \hline
   15 & 10 & 8 & $\ast$ & 10 & $\ast$ & $\ast$ & 28 &   &   &   &   &   &   &   &   &   \\ \hline
   17 & 6 & 0 & -9 & -6 & 12 & 12 & 9 & 36 &   &   &   &   &   &   &   &   \\ \hline
   19 & 15 & -3 & 18 & 6 & 15 & 6 & 3 & 18 & 45 &   &   &   &   &   &   &   \\ \hline
   21 & 10 & 17 & $\ast$ & 1 & -10 & $\ast$ & $\ast$ & 17 & $\ast$ & 55 &   &   &   &   &  &   \\ \hline
   23 & 21 & 15 & 15 & -18 & -9 & 21 & -6 & 9 & 6 & 18 & 66 &   &   &   &   &   \\ \hline
   25 & 15 & 3 & -3 & 30 & $\ast$ & -15 & 3 & -3 & 0 & $\ast$ & 30 & 78 &   &   &   &   \\ \hline
   27 & 28 & -1 & $\ast$ & 10 & -10 & $\ast$ & 28 & 26 & $\ast$ & 1 & 26 & $\ast$ & 91 &   &   &   \\ \hline
   29 & 21 & 27 & -12 & 3 & -30 & 0 & -21 & 24 & -3 & 24 & 27 & 12 & 30 & 105 &   &   \\ \hline
   31 & 36 & 24 & 33 & 24 & 45 & -12 & 6 & 36 & 12 & 6 & 30 & 30 & 33 & 45 & 120 &  \\ \hline
   33 & 28 & 8 & $\ast$ & -8 & 26 & $\ast$ & 19 & -28 & $\ast$ & -17 & $\ast$ & $\ast$  & 19 & 17 & $\ast$ & 136   \\ \hline 
  \end{tabular}
  \caption{$\frac{b}{4}Q(a;b)$}
\end{table}
\end{center}

The next problem is to generalize our results to the elliptic Dedekind-Aposotol sum \cite{FY}
\begin{equation}
s_{N+1,\tau}(a;b)
   :=
   \frac{1}{4b}\sum_{\substack{\mu,\nu =0 \\ (\mu,\nu ) \not=(0,0)}}^{b-1}
   (-1)^{\mu }
   \cs\left(2Ka\frac{\mu \tau +\nu }{b},k\right)
   \cs^{(N)}\left(2K\frac{\mu \tau +\nu }{b},k\right) \quad (n \in \mathbb{Z}_{\geq 0}), 
\end{equation}
where $\cs^{(N)}(z)$ is the $N$-th derivative of the $\cs(z)$.  
We remark that $s_{2n,\tau}(a;b)$ is identically zero. 
The elliptic Dedekind-Aposotol sum $s_{2n+1,\tau}(a;b)$ has the following properties, similar to the elliptic classical Dedekind sum $s_{\tau}(a;b)$. \\
(1) {\bf{parity}}
\begin{equation}
\label{eq:parity}
s_{2n+1,\tau}(-a;b)=-s_{2n+1,\tau}(a;b).
\end{equation}
(2) {\bf{even reduction}}
\begin{equation}
\label{eq:even reduction}
s_{2n+1,\tau}(a+2b;b)=s_{2n+1,\tau}(a;b).
\end{equation}
(3) {\bf{reciprocity}} (Fukuhara-Yui) If $a+b$, 
\begin{align}
\label{eq:elliptic reciprocity}
s_{2n+1,\tau}(a;b)+s_{2n+1,\tau}(b;a)
   &=
   R_{2n+1,0}(a,b)g_{2n+1}(k)
   -\sum_{l=1}^{\left\lfloor \frac{n+1}{2}\right\rfloor}
   R_{2n+1,l}(a,b)g_{2n+1-l}(k)g_{2l-1}(k),
\end{align}
where $g_{2n+1}(k)$ is the coefficient of Laurent expansion for $\cs(z,k)$ at $z=0$
$$
\cs(z,k)=\frac{1}{z}+\sum_{n=0}^{\infty}g_{2n+1}(k)z^{2n+1}
$$
and 
$$
R_{2n+1,l}(a,b)
   :=
   \frac{(2n)!}{4}
   \left(a^{2l-1}b^{2n+1-2l}+a^{2n+1-2l}b^{2l-1}+\frac{2n+1}{ab}\delta _{l,0} -a^{n}b^{n}\delta _{n,2l-1}\right).
$$
It is easy to show that if $a+b$ is odd then there exists a rational number $Q_{2n+1,l}(a;b)$ independent of $k$ such that 
\begin{align}
s_{2n+1,\tau}(a;b)
   &=
   Q_{2n+1,0}(a;b)g_{2n+1}(k)
   +\sum_{l=1}^{\left\lfloor \frac{n+1}{2}\right\rfloor}
   Q_{2n+1,l}(a;b)g_{2n+1-l}(k)g_{2l-1}(k). \nonumber 
\end{align}
Our main result is $Q_{1,0}(a;b)=6\left(s(a;2b)+s(2a;b)-2s(a;b)\right)$, and corresponds to the $n=0$ case.  
We desire to obtain explicit formulas using Dedekind-Apostol sums 
\begin{align}
s_{2n+1}(a;b)&:=
       \frac{1}{4b}\sum_{\nu =1}^{b-1}
       \cot{\left(\frac{\pi a\nu }{b}\right)}
       \cot^{(2n)}{\left(\frac{\pi \nu }{b}\right)}, \nonumber \\
s_{2n+1}(a;1)&:=
       0 \nonumber
\end{align}
for $Q_{2n+1,l}(a,b)$ $(n>1)$. 

\bibliographystyle{amsplain}

\noindent 
Department of Mathematics, Graduate School of Science, Kobe University, \\
1-1, Rokkodai, Nada-ku, Kobe, 657-8501, JAPAN\\
E-mail: g-shibukawa@math.kobe-u.ac.jp

\end{document}